\documentclass[amstex,12pt, amssymb]{article}

\usepackage{mathtext}
\usepackage[cp1251]{inputenc}
\usepackage[T2A]{fontenc}
\usepackage[dvips]{graphicx}
\usepackage{amsmath}
\usepackage{amssymb}
\usepackage{amsxtra}
\usepackage{latexsym}
\usepackage{ifthen}

\newcounter{lemma}[section]

\newcounter{corollary}[section]

\newcounter{remark}[section]

\newcounter{theorem}[section]

\newcounter{proposition}[section]

\newcounter{example}

\numberwithin{equation}{section}

\begin{document}

\markboth{V. DESYATKA, E.~SEVOST'YANOV}{\centerline{ON
CARATH\'{E}ODORY THEOREM ...}}

\def\cc{\setcounter{equation}{0}
\setcounter{figure}{0}\setcounter{table}{0}}

\overfullrule=0pt


\author{ZARINA KOVBA, EVGENY SEVOST'YANOV}

\title{
{\bf ON CARATH\'{E}ODORY THEOREM FOR OPEN DISCRETE UNCLOSED
MAPPINGS}}

\date{\today}
\maketitle

\begin{abstract}
We study mappings satisfying the inverse Poletsky-type inequality in
a domain of the Euclidean space. Such inequalities are well known
and play an important role in the study of quasiconformal and
quasiregular mappings. We consider the case when the mapped domain,
generally speaking, is not locally connected on its boundary. At the
same time, we consider the situation when the mapping is open and
discrete, but may not preserve the boundary of the domain. In terms
of prime ends, we obtain results on the equicontinuity of families
of such mappings in the closure of the definition domain. As a
consequence, we also obtain the corresponding statement for
Orlicz-Sobolev classes.
\end{abstract}


\section{Introduction}

A classical result of Carath\'{e}odory asserts the possibility of a
continuous extension to the boundary of conformal mappings in terms
of prime ends. In particular, the following theorem holds; see, for
example, \cite[Theorem~9.4]{CL}.

\medskip
{\bf Theorem~{\bf A.}}{\it\,A conformal mapping $w=f(z)$ of an open
disk $U$ onto a simply connected domain $D$ induces a one-to-one
correspondence between points $e^{i\theta}$ of the unit disk $K$ and
prime ends $P(e^{i\theta})$ of $D.$ Besides that, $C(f,
e^{i\theta})=I(P(e^{i\theta})).$ (Here, as usual, $C(f,
e^{i\theta})$ denotes the cluster set of $f$ at $e^{i\theta},$ and
$I(P(e^{i\theta}))$ is the impression of $P(e^{i\theta})$).}

\medskip
This result also has a fairly well-known generalization in the space
of quasiconformal mappings, see~\cite[Theorem~4.1]{Na}.

\medskip
{\bf Theorem~{\bf B.}}{\it\, Under a quasiconformal mapping $f$ of a
collared domain $D_0$ onto a domain $D,$ there exists a one-to-one
correspondence between the boundary points of $D_0$ and the prime
ends of $D.$ Moreover, the cluster set $C(f, b),$ $b\in \partial
D_0,$ coincides with the impression $I(P)$ of the corresponding
prime end $P$ of $D.$}

\medskip
In some cases, one can say that a continuous correspondence to the
boundary is even equicontinuous with respect to the points of the
boundary, see, for example, ~\cite[Theorem~3.1]{NP}.

\medskip
{\bf Theorem~{\bf С.} (N\"{a}kki--Palka).}{\it\, Let $\frak{F}$ be a
family of $K$-quasiconformal mappings of a domain $D\ne
\overline{{\Bbb R}^n}$ onto a domain $D^{\,\prime}$ and let either
$D$ or $D^{\,\prime}$ be quasiconformally collared on the boundary.
Then $\frak{F}$ is uniformly equicontinuous if and only if each
$f\in\frak{F}$ can be extended to a continuous mapping of
$\overline{D}$ onto $\overline{D^{\,\prime}}$ and
$\inf\limits_{\frak{F}} h(f(A))>0$ for some continuum $A$ in $D.$}

\medskip Here $h(f(A))$ denotes the chordal diameter of the set
$f(A)\subset\overline{{\Bbb R}^n},$ see
e.g.~\cite[Definition~12.1]{Va}. Unfortunately, we are not aware of
similar results on equicontinuity in terms of prime ends for
quasiregular mappings. At the same time, the second co-author have
proved the latter result for mappings more general than
quasiconformal in several different situations. In particular, it
was considered the case of homemorphisms, and, some later, open
discrete and closed mappings (\cite{ISS}, \cite{Sev$_1$}). The aim
of the present paper is to establish an analogue of Theorems~{\bf
A--C} for non-closed mappings and in the case of prime ends. In
particular, we obtain a general result that guarantees the
equicontinuity of a certain family in the neighborhood of a boundary
point. We emphasize that investigating the case of unclosed mappings
is extremely difficult, since all currently available results, in
one way or another, involve a technique that assumes that the
boundary preserves under mappings. In particular, path lifting
approach, used by us and some another authors, assumes boundary
preserving, as a rule (see, e.g., \cite{Cr$_1$}--\cite{Cr$_2$},
\cite{Sev$_1$}, \cite{Vu}). In \cite{DS} we have proved similar
result for ``good domains'', but for points, not prime ends.  This
manuscript is devoted to the ``prime ends case''.

\medskip
Let us recall some definitions. Let $y_0\in {\Bbb R}^n,$
$0<r_1<r_2<\infty$ and
\begin{equation}\label{eq1**}
A=A(y_0, r_1,r_2)=\left\{ y\,\in\,{\Bbb R}^n:
r_1<|y-y_0|<r_2\right\}\,.\end{equation}
Given sets $E,$ $F\subset\overline{{\Bbb R}^n}$ and a domain
$D\subset {\Bbb R}^n$ we denote by $\Gamma(E,F,D)$ a family of all
paths $\gamma:[a,b]\rightarrow \overline{{\Bbb R}^n}$ such that
$\gamma(a)\in E,\gamma(b)\in\,F $ and $\gamma(t)\in D$ for $t \in
(a, b).$ If $f:D\rightarrow {\Bbb R}^n,$ $y_0\in \overline{f(D)}$
and $0<r_1<r_2<d_0=\sup\limits_{y\in f(D)}|y-y_0|,$ then by
$\Gamma_f(y_0, r_1, r_2)$ we denote the family of all paths $\gamma$
in $D$ such that $f(\gamma)\in \Gamma(S(y_0, r_1), S( y_0, r_2),
A(y_0,r_1,r_2)).$ Let $Q:{\Bbb R}^n\rightarrow [0, \infty]$ be a
Lebesgue measurable function. We say that {\it $f$ satisfies
Poletsky inverse inequality} at the point $y_0\in \overline{f(D)},$
if the relation
\begin{equation}\label{eq2*A}
M(\Gamma_f(y_0, r_1, r_2))\leqslant \int\limits_{A(y_0,r_1,r_2)\cap
f(D)} Q(y)\cdot \eta^n (|y-y_0|)\, dm(y)
\end{equation}
holds for any $0<r_1<r_2<\infty$ and any Lebesgue measurable
function $\eta: (r_1,r_2)\rightarrow [0,\infty ]$ such that
\begin{equation}\label{eqA2}
\int\limits_{r_1}^{r_2}\eta(r)\, dr\geqslant 1\,.
\end{equation}
Observe that, all quasiconformal mappings and also quasiregular
mappings satisfy the relation~(\ref{eq2*A}) with some function~$Q,$
see sections 1.11 and 1.12 in \cite{Sev$_2$}.

\medskip
\begin{remark}\label{rem1}
Note that all quasiregular mappings $f:D\rightarrow {\Bbb R}^n$
satisfy the condition
\begin{equation}\label{eq22}
M(\Gamma_f(y_0, r_1, r_2))\leqslant \int\limits_{f(D)\cap
A(y_0,r_1,r_2)} K_O\cdot N(y, f, D)\cdot \eta^n (|y-y_0|)\, dm(y)
\end{equation}
at each point $y_0\in \overline{f(D)}\setminus\{\infty\}$ with some
constant $K_O=K_O(f)\geqslant 1$ and an arbitrary
Lebesgue-dimensional function $\eta: (r_1,r_2)\rightarrow
[0,\infty]$ which satisfies condition~(\ref{eqA2}). Indeed,
quasiregular mappings satisfy the condition
\begin{equation}\label{eq24}
M(\Gamma_f(y_0, r_1, r_2))\leqslant \int\limits_{f(D)\cap
A(y_0,r_1,r_2)} K_O\cdot N(y, f, D)\cdot (\rho^{\,\prime})^n(y)\,
dm(y)
\end{equation}
for an arbitrary function $\rho^{\,\prime}\in{\rm adm}\,
f(\Gamma_f(y_0, r_1, r_2)),$ see \cite[Remark~2.5.II]{Ri}. Put
$\rho^{\,\prime}(y):=\eta(|y-y_0|)$ for $y\in A(y_0,r_1,r_2)\cap
f(D),$ and $\rho^{\,\prime}(y)=0$ otherwise. By Luzin theorem, we
may assume that the function $\rho^{\,\prime}$ is Borel measurable
(see, e.g., \cite[Section~2.3.6]{Fe}). Due
to~\cite[Theorem~5.7]{Va},
$$\int\limits_{\gamma_*}\rho^{\,\prime}(y)\,|dy|\geqslant
\int\limits_{r_1}^{r_2}\eta(r)\,dr\geqslant 1$$
for each (locally rectifiable) path $\gamma_*$ in $\Gamma(S(y_0,
r_1), S(y_0, r_2), A(y_0, r_1, r_2)).$ By substituting the function
$\rho^{\,\prime}$ mentioned above into~(\ref{eq24}), we obtain the
desired ratio~(\ref{eq22}).
\end{remark}

\medskip
Recall that a mapping $f:D\rightarrow {\Bbb R}^n$ is called {\it
discrete} if the pre-image $\{f^{-1}\left(y\right)\}$ of each point
$y\,\in\,{\Bbb R}^n$ consists of isolated points, and {\it is open}
if the image of any open set $U\subset D$ is an open set in ${\Bbb
R}^n.$ Later, in the extended space $\overline{{{\Bbb R}}^n}={{\Bbb
R}}^n\cup\{\infty\}$ we use the {\it spherical (chordal) metric}
$h(x,y)=|\pi(x)-\pi(y)|,$ where $\pi$ is a stereographic projection
$\overline{{{\Bbb R}}^n}$ onto the sphere
$S^n(\frac{1}{2}e_{n+1},\frac{1}{2})$ in ${{\Bbb R}}^{n+1},$ namely,
\begin{equation*}\label{eq3C}
h(x,\infty)=\frac{1}{\sqrt{1+{|x|}^2}}\,,\qquad
h(x,y)=\frac{|x-y|}{\sqrt{1+{|x|}^2} \sqrt{1+{|y|}^2}}\,, \quad x\ne
\infty\ne y
\end{equation*}
(see \cite[Definition~12.1]{Va}). Further, the closure
$\overline{A}$ and the boundary $\partial A$ of the set $A\subset
\overline{{\Bbb R}^n}$ we understand relative to the chordal metric
$h$ in $\overline{{\Bbb R}^n}.$

\medskip
The boundary of $D$ is called {\it weakly flat} at the point $x_0\in
\partial D,$ if for every $P>0$ and for any neighborhood $U$
of the point $x_0$ there is a neighborhood $V\subset U$ of the same
point such that $M(\Gamma(E, F, D))>P$ for any continua $E, F\subset
D$ such that $E\cap\partial U\ne\varnothing\ne E\cap\partial V$ and
$F\cap\partial U\ne\varnothing\ne F\cap\partial V.$ The boundary of
$D$ is called weakly flat if the corresponding property holds at any
point of the boundary $D.$ Given a mapping $f:D\rightarrow {\Bbb
R}^n$, we denote
\begin{equation*}\label{eq1_A_4} C(f, x):=\{y\in \overline{{\Bbb
R}^n}:\exists\,x_k\in D: x_k\rightarrow x, f(x_k) \rightarrow y,
k\rightarrow\infty\}
\end{equation*}
and
\begin{equation*}\label{eq1_A_5} C(f, \partial
D)=\bigcup\limits_{x\in \partial D}C(f, x)\,.
\end{equation*}
In what follows, ${\rm Int\,}A$ denotes the set of inner points of
the set $A\subset \overline{{\Bbb R}^n}.$ Recall that the set
$U\subset\overline{{\Bbb R}^n}$ is a neighborhood of the point
$z_0,$ if $z_0\in {\rm Int\,}A.$

\medskip
Consider the following definition, which goes back to
N\"akki~\cite{Na}, cf.~\cite{KR}. The boundary of a domain $D$ in
${\Bbb R}^n$ is said to be {\it locally quasiconformal} if every
$x_0\in\partial D$ has a neighborhood $U$ that admits a
quasiconformal mapping $\varphi$ onto the unit ball ${\Bbb
B}^n\subset{\Bbb R}^n$ such that $\varphi(\partial D\cap U)$ is the
intersection of ${\Bbb B}^n$ and a coordinate hyperplane. The
sequence of cuts $\sigma_m,$ $m=1,2,\ldots ,$ is called {\it
regular,} if
$\overline{\sigma_m}\cap\overline{\sigma_{m+1}}=\varnothing$ for
$m\in {\Bbb N}$ and, in addition, $d(\sigma_{m})\rightarrow 0$ as
$m\rightarrow\infty.$ If the end $K$ contains at least one regular
chain, then $K$ will be called {\it regular}. We say that a bounded
domain $D$ in ${\Bbb R}^n$ is {\it regular}, if $D$ can be
quasiconformally mapped to a domain with a locally quasiconformal
boundary whose closure is a compact in ${\Bbb R}^n,$ and, besides
that, every prime end in $D$ is regular. Note that, the space
$\overline{D}_P=D\cup E_D$ is metric, which can be demonstrated as
follows. If $g:D_0\rightarrow D$ is a quasiconformal mapping of a
domain $D_0$ with a locally quasiconformal boundary onto some domain
$D,$ then for $x, y\in \overline{D}_P$ we put:
\begin{equation}\label{eq5M}
\rho(x, y):=|g^{\,-1}(x)-g^{\,-1}(y)|\,,
\end{equation}
where the element $g^{\,-1}(x),$ $x\in E_D,$ is to be understood as
some (single) boundary point of the domain $D_0.$ The specified
boundary point is unique and well-defined by~\cite[Theorem~2.1,
Remark~2.1]{IS}, cf.~\cite[Theorem~4.1]{Na}. It is easy to verify
that~$\rho$ in~(\ref{eq5M}) is a metric on $\overline{D}_P.$

\medskip
Assume that, $D^{\,\prime}\setminus C(f, \partial
D)=\bigcup\limits_{i=1}^ND_i,$ where $D_i$ is a regular domain for
$1\leqslant i\leqslant N,$ $D_i\cap D_j=\varnothing$ for $i\ne j.$
Let $\rho_i$ is a metric in $\overline{D_i}_P$ defined
by~(\ref{eq5M}). We set $\rho_j(x, y)=0$ for $x, y\in
\bigcup\limits_{i=1}^N\overline{D_i}_P\setminus \overline{D_j}_P,$
and $\rho_j(x, y)=1$ for $x\in \overline{D_j}_P$ and $y\not\in
\overline{D_j}_P.$ Then $\rho_j$ is a pseudometric on
$\bigcup\limits_{i=1}^N\overline{D_i}_P$ (see
\cite[2.2.21.XV]{Ku$_1$}). Set
\begin{equation}\label{eq1}
\rho(x, y)=\sum\limits_{i=1}^N2^{\,-i}\frac{\rho_i(x,
y)}{1+\rho_i(x, y)}\,.
\end{equation}
Observe that, $\rho^{\,*}_j(x, y):=\frac{\rho_j(x, y)}{1+\rho_j(x,
y)}$ is a metric on $\overline{D_j}_P$ for $1\leqslant j\leqslant N$
(see \cite[21.2.V]{Ku$_1$}). In addition, the space
$(\overline{D_{j}}_P, \rho^{\,*}_j)$ which is homeomorphic to
$(\overline{D_{j}}_P, \rho_j)$ (see \cite[21.2.V]{Ku$_1$}). Besides
that, $\rho(x, y)$ is a metric on
$\bigcup\limits_{i=1}^N\overline{D_i}_P,$ see
\cite[Remark~2.2.21.XV]{Ku$_1$}.

\medskip
\begin{remark}\label{rem3}
Observe that, the space $\bigcup\limits_{i=1}^N\overline{D_i}_P$ is
compact. Indeed, if $x_k\in \bigcup\limits_{i=1}^N\overline{D_i}_P,$
then there is $1\leqslant j_0\leqslant N$ and a subsequence
$x_{k_l},$ $l=1,2,\ldots ,$ such that $x_{k_l}\in
\overline{D_{j_0}}_P$ for any $l\in {\Bbb N}.$ By the definition of
the prime end space, we may find a sequence $y_l\in D_{j_0},$
$l=1,2,\ldots ,$ such that $\rho_{j_0}(y_l, x_{k_l})<1/l$ for $l\in
{\Bbb N}.$ Since by the definition the domain $D_{j_0}$ is bounded,
we may consider that $y_l\rightarrow y_0\in \overline{D_{j_0}}.$
There are two cases: $y_0\in D_{j_0},$ or $y_0\in \partial D_{j_0}.$
In the first case, when $y_0\in D_{j_0},$ by the definition of
$\rho_{j_0}$ in~(\ref{eq5M}), we have that $\rho_{j_0}(y_l,
y_0)\rightarrow 0$ as $l\rightarrow \infty.$ Now, by the triangle
inequality $\rho_{j_0}(x_{k_l}, y_0)\leqslant \rho_{j_0}(x_{k_l},
y_l)+\rho_{j_0}(y_l, y_0)<\frac{1}{l}+\rho_{j_0}(y_l,
y_0)\rightarrow 0$ as $l\rightarrow\infty.$ Thus,
$x_{k_l}\rightarrow y_0$ as $l\rightarrow \infty$ in the
metric~$\rho_{j_0}.$ By the definition of the metric $\rho$
in~(\ref{eq1}) $x_{k_l}\rightarrow y_0$ as $l\rightarrow \infty$ in
the metric~$\rho,$ as well. In the second case, when $y_0\in
\partial D_{j_0},$ we consider the homeomorphism $g:D_0\rightarrow
D_{j_0}$ of a domain $D_0$ with a locally quasiconformal boundary
onto $D_{j_0}.$ Such a domain $D_0$ and a homeomorphism $g$ exist by
the definition of a regular domain. Since $D_0$ is bounded, we may
consider that $g^{\,-1}(x_{k_l})\rightarrow z_0\in \partial D_0$ as
$l\rightarrow\infty.$ Set $P_0=g(z_0)\in E_{D_{j_0}}.$ The prime end
$P_0$ is well-defined due to~\cite[Theorem~2.1, Remark~2.1]{IS},
cf.~\cite[Theorem~4.1]{Na}. In addition, by the definition of the
metrics $\rho_{j_0},$ $\rho_{j_0}(x_{k_l}, P_0)\rightarrow 0$ as
$l\rightarrow\infty.$ By the definition of $\rho$ in~(\ref{eq1}),
$x_{k_l}\rightarrow P_0$ as $l\rightarrow \infty$ by the
metric~$\rho,$ as well.
\end{remark}

Given a Lebesgue measurable function $Q:{\Bbb R}^n\rightarrow [0,
\infty],$ $r>0$ and a point $y_0\in {\Bbb R}^n$ we set
$\widetilde{Q}(y)=\max\{Q(y), 1\}$ and
\begin{equation}\label{eq2B}
\widetilde{q}_{y_0}(r)=\frac{1}{\omega_{n-1}r^{n-1}}
\int\limits_{S(y_0, r)}\widetilde{Q}(y)\,d\mathcal{H}^{n-1}\,.
\end{equation}
Given a number $\delta>0,$ domains $D, D^{\,\prime}\subset{\Bbb
R}^n,$ $n\geqslant 2,$ a set $E\subset \overline{D^{\,\prime}},$
closed in $\overline{{\Bbb R}^n},$ such that $\partial
D^{\,\prime}\subset E,$ a set $E_0\subset D,$ a finite set of points
$P=\{a_i\}_{i=1}^{S}\subset D^{\,\prime}\setminus E,$ $0\leqslant
i\leqslant S,$ and a Lebesgue measurable function $Q:{\Bbb
R}^n\rightarrow [0, \infty],$ denote by ${\frak S}^P_{E, E_0 \delta,
Q}(D, D^{\,\prime})$ a family of all open discrete mappings $f$ of
$D$ onto $D^{\,\prime}$ satisfying the
conditions~(\ref{eq2*A})--(\ref{eqA2}) for all $0<r_1<r_2<\infty$
each $y_0\in \overline{D^{\,\prime}}$ such that

\medskip
1) $C(f, \partial D)\subset E,$

\medskip
2) $f^{\,-1}(E\cap D^{\,\prime})\subset E_0,$

\medskip
3) $h(f^{\,-1}(a_i), \partial D)\geqslant \delta$ for all $i\in
{\Bbb N}$ (here $h(A, B)$ denotes the chordal distance between sets
in $\overline{{\Bbb R}^n},$ $h(A, B)=\inf\limits_{x\in A, y\in
B}h(x, y)$).

The following theorem in some versions (in particular, for
homeomorphisms and closed maps) was proved by the second co-author,
see, e.g., \cite{SSD}, \cite{SevSkv$_1$}, \cite{SevSkv$_2$}, and
\cite{SevSkv$_3$}. The result formulated below takes into account
the most general situation, when the mapping is neither a
homeomorphism nor closed (=boundary preserving).

\medskip
\begin{theorem}\label{th2}
{\it\, Let $D$ and $D^{\,\prime}$ be domains in ${\Bbb R}^n,$
$n\geqslant 2,$ let $D$ be a domain with a weakly flat boundary and
let $E$ be a set in $\overline{D^{\,\prime}},$ which is closed in
$\overline{{\Bbb R}^n}$ and such that $\partial D^{\,\prime}\subset
E.$ Suppose that the following conditions are fulfilled:

\medskip
1) for each point $y_0\in \partial D^{\,\prime}\setminus\{\infty\}$
either

\medskip
1a) there is $0<r_0<\sup\limits_{y\in D^{\,\prime}}|y-y_0|$ such
that, for each $0<r_1<r_2<r_0$ there exists a set $E_1\subset[r_1,
r_2]$ of a positive linear Lebesgue measure is that the function $Q$
is integrable on $S(y_0, r)$ for each $r\in E_1;$ or

\medskip
1b) \begin{equation}\label{eq6A}
\int\limits_{0}^{\delta(y_0)}\frac{dt}{t
\widetilde{q}_{y_0}^{\,\frac{1}{n-1}}(t)}=\infty\,;
\end{equation}

\medskip
2) $D^{\,\prime}\setminus E$ consists of finite number of pairwise
disjoint domains $D_1, D_2,\ldots, D_N$ any of which is regular and,
besides that, for any $1\leqslant i\leqslant N$ there is $a_i\in P$
such that $a_i\in D_i;$

\medskip
3) the set $E_0$ is nowhere dense in $D.$

\medskip
Then the family ${\frak S}^P_{E, E_0 \delta, Q}(D, D^{\,\prime})$ is
equicontinuous at every point $x_0\in \partial D$ with respect to
$D\setminus E_0,$ i.e., for every $\varepsilon>0$ there is
$\delta=\delta(x_0, \varepsilon)>0$ such that $\rho(f(x),
f(y))<\varepsilon$ for any $x, y\in B(x_0, \delta)\cap (D\setminus
E_0)$ and any $f\in {\frak S}^P_{E, E_0 \delta, Q}(D,
D^{\,\prime}),$ where the metric $\rho$ is defined in~(\ref{eq1}).
In particular, any $f\in {\frak S}^P_{E, E_0 \delta, Q}(D,
D^{\,\prime})$ has a continuous extension to $x_0$ with respect to
$D\setminus E_0$ in terms of the space
$\bigl(\bigcup\limits_{i=1}^N\overline{D_i}_P, \rho\bigr).$ }
\end{theorem}

\medskip
\begin{remark}\label{rem2}
The condition 1) of Theorem~\ref{th2} may be replaced on a simpler
one: $Q\in L^1(D^{\,\prime}).$ Indeed, let $0<r_0<\sup\limits_{y\in
D^{\,\prime}}|y-y_0|.$ We may assume that $Q$ is extended by zero
outside $D^{\,\prime}.$ By the Fubini theorem (see, e.g.,
\cite[Theorem~8.1.III]{Sa}) we obtain that
$$\int\limits_{r_1<|y-y_0|<r_2}Q(y)\,dm(y)=\int\limits_{r_1}^{r_2}
\int\limits_{S(y_0, r)}Q(y)\,d\mathcal{H}^{n-1}(y)dr<\infty\,.$$
This means the fulfillment of condition~1) in
Theorem~\ref{th2}.~$\Box$
\end{remark}

\section{Preliminaries}

The following statement holds, see, e.g.,
\cite[Theorem~1.I.5.46]{Ku$_2$}).

\medskip
\begin{proposition}\label{pr2}
{\it\, Let $A$ be a set in a topological space $X.$ If the set $C$
is connected and $C\cap A\ne \varnothing\ne C\setminus A,$ then
$C\cap
\partial A\ne\varnothing.$}
\end{proposition}

\medskip
The following lemma was proved in \cite[Lemma~5.1]{DS}.

\medskip
\begin{lemma}\label{lem3}
{\it\, Let $D$ be a domain in ${\Bbb R}^n,$ $n\geqslant 2,$ and let
the points $a\in D, b\in \overline{D},$ and $c\in D, d\in
\overline{D}$ be such that there are paths $\gamma_1:[0,
1]\rightarrow \overline{D}$ and $\gamma_2:[0, 1]\rightarrow
\overline{D},$ such that $\gamma_i(t)\in D$ for all $t\in (0, 1),$
$i=1,2,$ $\gamma_1(0)=a,$ $\gamma_1(1)=b,$ $\gamma_2(0)= c,$
$\gamma_2(1)=d.$ Then there are similar disjoint paths
$\gamma^{\,*}_1$ and $\gamma^{\,*}_2.$

Moreover, if $\gamma_1$ is a Jordan path, then we may take
$\gamma_1$ as $\gamma^{\,*}_1,$ and as $\gamma^{\,*}_2$ we may take
a path which coincides with $\gamma_2$ on some segments $[0, t_0]$
and $[T_0, 1],$ where $0<t_0<T_0<1.$}
\end{lemma}

\medskip The following lemma was proved in \cite[Lemma~2.1]{ISS} for
regular domains $D,$ cf. \cite[Lemma~12.1]{Sev$_2$}. We need it in
more general situation.

\medskip
\begin{lemma}\label{lem1}{\it\,
Let $D^{\,\prime}\subset {\Bbb R}^n,$ $n\geqslant 2,$ let $E$ be a
set in $\overline{D^{\,\prime}},$ which is closed in
$\overline{{\Bbb R}^n}$ and such that $\partial D^{\,\prime}\subset
E$ and that $D^{\,\prime}\setminus E$ consists of finite number of
pairwise disjoint domains $D_1, D_2,\ldots, D_N$ any of which is
regular. Let $1\leqslant j_1, j_2\leqslant N$ and let
$x_m\rightarrow P_1\in {\overline{D}_{j_1}}_P,$ $y_m\rightarrow
P_2\in {\overline{D}_{j_2}}_P$ as $m\rightarrow\infty,$ $P_1\ne
P_2.$ Suppose that $d_m, g_m,$ $m=1,2,\ldots,$ are sequences of
descending domains, corresponding to $P_1$ and $P_2,$ $d_1\cap
g_1=\varnothing,$ and $x_0\in D_{j_1}\setminus (d_1\cup g_1),$
$y_0\in D_{j_2}\setminus (d_1\cup g_1).$ Then there are arbitrarily
large $k_0\in {\Bbb N},$ $M_0=M_0(k_0)\in {\Bbb N}$ and
$0<t_1=t_1(k_0), t_2=t_2(k_0)<1$ for which the following condition
is fulfilled: for each $m\geqslant M_0$ there are
disjoint paths
\begin{equation}\label{eq1A}
\gamma_{1,m}(t)=\quad\left\{
\begin{array}{rr}
\widetilde{\alpha}(t), & t\in [0, t_1],\\
\widetilde{\alpha_m}(t), & t\in [t_1, 1]\end{array}
\right.\,,\quad\gamma_{2,m}(t)=\quad\left\{
\begin{array}{rr}
\widetilde{\beta}(t), & t\in [0, t_2],\\
\widetilde{\beta_m}(t), & t\in [t_2, 1]\end{array}\,, \right.
\end{equation}
such that:

1) $\gamma_{1, m}(0)=x_0,$ $\gamma_{1, m}(1)=x_m,$ $\gamma_{2,
m}(0)=y_0$ and $\gamma_{2, m}(1)=y_m;$

2) $|\gamma_{1, m}|\cap \overline{g_{k_0}}=\varnothing=|\gamma_{2,
m}|\cap \overline{d_{k_0}};$

3) $\widetilde{\alpha_m}(t)\in d_{k_0+1}$ for $t\in [t_1, 1]$ and
$\widetilde{\beta_m}(t)\in g_{k_0+1}$ for $t\in [t_2, 1]$ (see
Figure~\ref{fig3}).
\begin{figure}[h]
\centering\includegraphics[width=200pt]{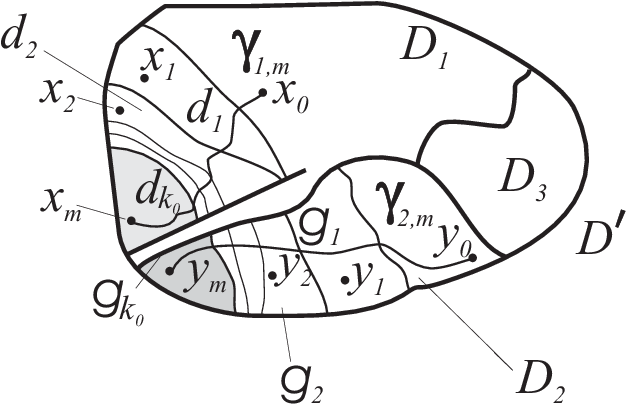} \caption{To
the statement of Lemma~\ref{lem1}}\label{fig3}
\end{figure}
}
\end{lemma}

\medskip
\begin{proof}
There are two cases: $j_1=j_2,$ or $j_1\ne j_2.$ In the first case,
$j_1=j_2,$ the proof was given in \cite[Lemma~2.1]{ISS}. In the
second case, $j_1\ne j_2,$ the proof is almost obvious. Indeed, fix
$k_0\in {\Bbb N}.$ Since $x_m\rightarrow P_1$ as
$m\rightarrow\infty,$ we may find $M=M(k_0)$ such that $x_m\in
d_{k_0}$ and $y_m\in g_{k_0}$ for $m\geqslant M=M(k_0).$ Let us join
the points $x_0$ and $x_{M(k_0)}$ by a path $\widetilde{\alpha}:[0,
t_1]\rightarrow D_{j_1}$ while $\alpha(0)=x_0$ and
$\widetilde{\alpha}(t_1)=x_{M(k_0)}.$ Since $d_m\subset d_{k_0}$ for
every $m\geqslant k_0,$ we may join the points $x_{M(k_0)}$ and
$x_m$ by a path $\widetilde{\alpha_m}$ inside $d_{k_0}.$ Now
consider a path $\gamma_{1, m}$ defined in~(\ref{eq1A}). Similarly,
we join the points $y_0$ and $y_{M(k_0)}$ by a path
$\widetilde{\beta}:[0, t_2]\rightarrow D_{j_2}$ while
$\widetilde{\beta}(0)=y_0$ and $\beta(t_2)=y_{M(k_0)}.$ Since
$g_m\subset g_{k_0}$ for every $m\geqslant k_0,$ we may join the
points $y_{M(k_0)}$ and $y_m$ by a path $\beta_m$ inside $g_{k_0}.$
Now consider a path $\gamma_{2, m}$ defined in~(\ref{eq1A}). Observe
that the paths $\gamma_{1, m}$ and $\gamma_{2, m}$ are required
because $D_{j_1}\cap D_{j_2}=\varnothing.$~$\Box$
\end{proof}

\medskip
The version of the following Lemma was proved in \cite{ISS} for
homeomorphisms and in \cite{Sev$_1$} for boundary preserving
mappings. Now we need it in the situation, when mappings may not
preserve the boundary of a domain.

\medskip
\begin{lemma}\label{lem4}
{\it Let $D$ and $D^{\,\prime}$ be domains in ${\Bbb R}^n,$
$n\geqslant 2,$ let $E$ be a set in $\overline{D^{\,\prime}},$ which
is closed in $\overline{{\Bbb R}^n}$ and such that $\partial
D^{\,\prime}\subset E$ and that $D^{\,\prime}\setminus E$ consists
of finite number of pairwise disjoint domains $D_1, D_2,\ldots, D_N$
any of which is regular. Assume that, $f$ is an open discrete
mapping of $D$ onto $D^{\,\prime}$ which satisfies the
relation~(\ref{eq2*A}) at every point
$y_0\in\overline{D^{\,\prime}},$ every $0<r_1<r_2<\infty$ and some a
Lebesgue measurable function $Q$ which satisfies condition 1) in
Theorem~\ref{th2}. Let also $d_m$ be a sequence of decreasing
domains which corresponds to a chain of cuts $\sigma_m,$
$m=1,2,\ldots, $ lying onto spheres $S(\overline{x_0}, r_m)$ such
that $\overline{x_0}\in
\partial D^{\,\prime},$ while $r_m\rightarrow 0$ as $m\rightarrow\infty.$
Then under conditions and notions of Lemma~\ref{lem1} we may find
$k_0\in {\Bbb N}$ for which there exists $0<N=N(k_0, Q,
D^{\,\prime})<\infty,$ which does not depend on $m$ and $f$ such
that
$$M(\Gamma_m)\leqslant N,\qquad m\geqslant M_0=M_0(k_0)\,,$$
where $\Gamma_m$ is a family of paths $\gamma:[0, 1]\rightarrow D$
in $D$ such that $f(\gamma)\in \Gamma(|\gamma_{1, m}|, |\gamma_{2,
m}|, D^{\,\prime}).$ }
\end{lemma}
\begin{proof}
Let $k_0$ ne an arbitrary number for which the statement of
Lemma~\ref{lem1} holds. By the definition of $\gamma_{1, m}$ and a
family $\Gamma_m$ we may write
\begin{equation}\label{eq7A}
\Gamma_m=\Gamma_m^1\cup \Gamma_m^2\,,
\end{equation}
where $\Gamma_m^1$ is a family of paths $\gamma\in\Gamma_m$ such
that $f(\gamma)\in \Gamma(|\widetilde{\alpha}|, |\gamma_{2, m}|,
D^{\,\prime})$ and $\Gamma_m^2$ is a family of paths
$\gamma\in\Gamma_m$ such that $f(\gamma)\in
\Gamma(|\widetilde{\alpha}_m|, |\gamma_{2, m}|, D^{\,\prime}).$

\medskip
Under the notions of Lemma~\ref{lem1}, let
$\varepsilon_0=\min\{{\rm dist}\,(|\widetilde{\alpha}|,
\overline{g_{k_0}}), {\rm dist}\,(|\widetilde{\alpha}|,
|\widetilde{\beta}|)\}>0.$ Due to the condition~1) of the theorem,
for any $y_0\in\overline{D^{\,\prime}}$ we may find
$\varepsilon=\varepsilon(y_0)>0,$ such that
$\varepsilon/4<\varepsilon/2<\widetilde{\varepsilon}/2$ and
\begin{equation}\label{eq1AA}
\int\limits_{\varepsilon/4}^{\varepsilon/2}\
\frac{dr}{r\widetilde{q}_{y_0}^{\frac{1}{n-1}}(r)}>0\,,\end{equation}
where $\widetilde{q}_{y_0}(r)$ is defined in~(\ref{eq2B}).
We consider now the cover of the set $|\widetilde{\alpha}|$ of the
following type: $\bigcup\limits_{y_0\in |\widetilde{\alpha}|}B(y_0,
\varepsilon(y_0)/4).$ Since $|\widetilde{\alpha}|$ is a compactum in
$D^{\,\prime},$ there is a number $N_0\in {\Bbb N}$ such that
$|\widetilde{\alpha}|\subset \bigcup\limits_{i=1}^{N_0} B(z_i,
\varepsilon^{\,i}/4),$ where $\varepsilon^{\,i}=\varepsilon(z_i),$
$z_i\in |\widetilde{\alpha}|$ for $1\leqslant i\leqslant N_0.$ Due
to~\cite[Theorem~1.I.5.46]{Ku$_2$}
\begin{equation}\label{eq5A}
\Gamma(|\widetilde{\alpha}|, |\gamma_{2, m}|,
D^{\,\prime})>\bigcup\limits_{i=1}^{N_0} \Gamma(S(z_i,
\varepsilon^{\,i}/4), S(z_i, \varepsilon^{\,i}/2), A(z_i,
\varepsilon^{\,i}/4, \varepsilon^{\,i}/2))\,.
\end{equation}
Fix $\gamma\in \Gamma_m^1,$ $\gamma:[0, 1]\rightarrow D,$
$\gamma(0)\in |\widetilde{\alpha}|,$ $\gamma(1)\in |\gamma_{2, m}|.$
It follows from the relation~(\ref{eq5A}) that $f(\gamma)$ has a
subpath $f(\gamma)_1:=f(\gamma)|_{[p_1, p_2]}$ such that
$$f(\gamma)_1\in \Gamma(S(z_i, \varepsilon^{\,i}/4), S(z_i,
\varepsilon^{\,i}/2), A(z_i, \varepsilon^{\,i}/4,
\varepsilon^{\,i}/2))$$ for some $1\leqslant i\leqslant N_0.$ Then
$\gamma|_{[p_1, p_2]}$ is a subpath which is, on one hand, is a
subpath of $\gamma,$ and on the other hand belongs to~$\Gamma_f(z_i,
\varepsilon^{\,i}/4, \varepsilon^{\,i}/2),$ because
$$f(\gamma|_{[p_1, p_2]})=f(\gamma)|_{[p_1, p_2]}\in\Gamma(S(z_i,
\varepsilon^{\,i}/4), S(z_i, \varepsilon^{\,i}/2), A(z_i,
\varepsilon^{\,i}/4, \varepsilon^{\,i}/2)).$$ Thus
\begin{equation}\label{eq6}
\Gamma_m^1>\bigcup\limits_{i=1}^{N_0} \Gamma_f(z_i,
\varepsilon^{\,i}/4, \varepsilon^{\,i}/2)\,.
\end{equation}
Let
\begin{equation*}\label{eq2}
I_i=I_i(z_i, \varepsilon^{\,i}/4,
\varepsilon^{\,i}/2)=\int\limits_{\varepsilon^{\,i}/4}
^{\varepsilon^{\,i}/2}\
\frac{dr}{r\widetilde{q}_{z_i}^{\frac{1}{n-1}}(r)}\,.
\end{equation*}
Note that, $I_i\ne 0$ because~(\ref{eq1AA}) holds. In addition,
$I_i\ne\infty,$ because $I_i\leqslant\log\frac{r_2}{r_1}<\infty,$
$i=1,2, \ldots, {N_0}.$
We set
$$\eta_i(r)=\begin{cases}
\frac{1}{I_ir\widetilde{q}_{z_i}^{\frac{1}{n-1}}(r)}\,,&
r\in [\varepsilon^{\,i}/4, \varepsilon^{\,i}/2]\,,\\
0,& r\not\in [\varepsilon^{\,i}/4, \varepsilon^{\,i}/2]\,.
\end{cases}$$
Note that the function~$\eta_i$ satisfies the condition
~$\int\limits_{\varepsilon^{\,i}/4}^{\varepsilon^{\,i}/2}\eta_i(r)\,dr=1,$
therefore, it may be substituted in the right-hand side
of~(\ref{eq2*A}). Now, due to~(\ref{eq6}) and by the subadditivity
of the modulus of families of paths (see \cite[Theorem~6.1]{Va}) we
obtain that
\begin{gather}
M(\Gamma_m^1)\leqslant \sum\limits_{i=1}^{N_0}M(\Gamma_f(z_i,
\varepsilon^{\,i}/4, \varepsilon^{\,i}/2))\nonumber\\
\label{eq2A} \leqslant\sum\limits_{i=1}^{N_0}\int\limits_{A(z_i,
\varepsilon^{\,i}/4, \varepsilon^{\,i}/2)}
\widetilde{Q}(y)\,\eta^n_i(|y-z_i|)\,dm(y)\,.
\end{gather}
However, by the Fubini theorem (see, e.g.,
\cite[Theorem~8.1.III]{Sa})
\begin{equation}\label{eq7CB}
\int\limits_{A(z_i, \varepsilon^{\,i}/4, \varepsilon^{\,i}/2)}
\widetilde{Q}(y)\,\eta^n_i(|y-z_i|)\,dm(y)
=\int\limits_{\varepsilon^{\,i}/4}^{\varepsilon^{\,i}/2}\int\limits_{S(z_i,
r)}\widetilde{Q}(y)\eta^n_i(|y-z_i|)\,d\mathcal{H}^{n-1}\,dr\,=
\end{equation}
$$=\frac{\omega_{n-1}}{I_i^n}\int\limits_{\varepsilon^{\,i}/4}
^{\varepsilon^{\,i}/2}r^{n-1} \widetilde{q}_{z_i}(r)\cdot
\frac{dr}{r^n\widetilde{q}^{\frac{n}{n-1}}_{z_i}(r)}=
\frac{\omega_{n-1}}{I_i^{n-1}}\,.$$
Combining~(\ref{eq2A}) with (\ref{eq7CB}), we obtain that
\begin{equation}\label{eq3A}
M(\Gamma_m^1)\leqslant
\sum\limits_{i=1}^{N_0}\frac{\omega_{n-1}}{I_i^{n-1}}\,,\qquad
m\geqslant M_0\,.
\end{equation}
Follow, by~\cite[Theorem~1.I.5.46]{Ku$_2$},
$\Gamma_m^2>\Gamma_f(\overline{x_0}, r_{k_0+1}, r_{k_0}).$
Arguing as above, we set
$$\eta_0(r)=\begin{cases}
\frac{1}{I_0r\widetilde{q}_{\overline{x_0}}^{\frac{1}{n-1}}(r)}\,,&
r\in [\varepsilon^{\,0}/4, \varepsilon^{\,0}/2]\,,\\
0,& r\not\in [\varepsilon^{\,0}/4, \varepsilon^{\,0}/2]\,,
\end{cases}$$
where $I_0=I_0(\overline{x_0}, \varepsilon^{\,0}/4,
\varepsilon^{\,0}/2)=\int\limits_{\varepsilon^{\,0}/4}^{\varepsilon^{\,0}/2}\
\frac{dr}{r\widetilde{q}_{\overline{x_0}}^{\frac{1}{n-1}}(r)}.$
Then from the latter relation we obtain that
\begin{equation}\label{eq4D}
M(\Gamma_m^2)\leqslant \frac{\omega_{n-1}}{I_0^{n-1}}\,,\quad
m\geqslant M_0\,.
\end{equation}
Thus, by the relations~(\ref{eq7A}), (\ref{eq3A}) and~(\ref{eq4D}),
due to the subadditivity of the modulus of families of paths, we
obtain that
\begin{equation}\label{eq3}
M(\Gamma_m)\leqslant
\sum\limits_{i=1}^{N_0}\frac{\omega_{n-1}}{I_i^{n-1}}+\frac{\omega_{n-1}}{I_0^{n-1}}<\infty\,,
\quad m\geqslant M_0\,.
\end{equation}
The right hand part of the latter relation does not depend on~$m,$
so we may set
$N:=\sum\limits_{i=1}^{N_0}\frac{\omega_{n-1}}{I_i^{n-1}}+\frac{\omega_{n-1}}{I_0^{n-1}}.$
Lemma~\ref{lem4} is totally proved.~$\Box$
\end{proof}

\medskip
Let $D\subset {\Bbb R}^n,$ $f:D\rightarrow {\Bbb R}^n$ be a discrete
open mapping, $\beta: [a,\,b)\rightarrow {\Bbb R}^n$ be a path, and
$x\in\,f^{\,-1}(\beta(a)).$ A path $\alpha: [a,\,c)\rightarrow D$ is
called a {\it maximal $f$-lifting} of $\beta$ starting at $x,$ if
$(1)\quad \alpha(a)=x\,;$ $(2)\quad f\circ\alpha=\beta|_{[a,\,c)};$
$(3)$\quad for $c<c^{\prime}\leqslant b,$ there is no a path
$\alpha^{\prime}: [a,\,c^{\prime})\rightarrow D$ such that
$\alpha=\alpha^{\prime}|_{[a,\,c)}$ and $f\circ
\alpha^{\,\prime}=\beta|_{[a,\,c^{\prime})}.$ Here and in the
following we say that a path $\beta:[a, b)\rightarrow
\overline{{\Bbb R}^n}$ converges to the set $C\subset
\overline{{\Bbb R}^n}$ as $t\rightarrow b,$ if $h(\beta(t),
C)=\sup\limits_{x\in C}h(\beta(t), C)\rightarrow 0$ as $t\rightarrow
b.$ The following is true (see~\cite[Lemma~3.12]{MRV}).

\medskip
\begin{proposition}\label{pr3}
{\it\, Let $f:D\rightarrow {\Bbb R}^n,$ $n\geqslant 2,$ be an open
discrete mapping, let $x_0\in D,$ and let $\beta: [a,\,b)\rightarrow
{\Bbb R}^n$ be a path such that $\beta(a)=f(x_0)$ and such that
either $\lim\limits_{t\rightarrow b}\beta(t)$ exists, or
$\beta(t)\rightarrow \partial f(D)$ as $t\rightarrow b.$ Then
$\beta$ has a maximal $f$-lifting $\alpha: [a,\,c)\rightarrow D$
starting at $x_0.$ If $\alpha(t)\rightarrow x_1\in D$ as
$t\rightarrow c,$ then $c=b$ and $f(x_1)=\lim\limits_{t\rightarrow
b}\beta(t).$ Otherwise $\alpha(t)\rightarrow \partial D$ as
$t\rightarrow c.$}
\end{proposition}

\medskip
The following lemma was proved in \cite[Lemma~5.3]{DS}.

\medskip
\begin{lemma}\label{lem2}
{\it\, Let $D^{\,\prime}$ be a domain in ${\Bbb R}^n,$ $n\geqslant
2,$ and let $E\subset \overline{D^{\,\prime}},$ $\partial
D^{\,\prime}\subset E,$ and let $E$ be closed in $\overline{\Bbb
R^{n}}.$ Suppose also that, for any $z_0\in E$ and for any
neighborhood $U$ of the point $z_0$ there is a neighborhood
$V\subset U$ of this point such that $(V\setminus E)\cap
D^{\,\prime}$ consists of a finite number of components. Let $z_1,
z_2\in E,$ $z_1\ne z_2,$ and let $x_k, y_k\in D^{\,\prime}\setminus
E,$ $k=1,2,\ldots ,$ be sequences, which converge to $z_1$ and
$z_2,$ respectively. Then there are subsequences $x_{k_l}$ and
$y_{k_l},$ $l=1,2,\ldots ,$ each of which belong to exactly one
(corresponding) component $K_1$ and $K_2$ of $D^{\,\prime}\setminus
E.$

Moreover, if $a_1\in K_1$ and $a_2\in K_2,$ $a_1\ne a_2,$ then there
are (possibly) new subsequences $x_{k_{l_m}}$ and $y_{k_{l_m}},$
$m=1,2,\ldots ,$ and paths $\alpha:[0, 1]\rightarrow
\overline{D^{\,\prime}},$ $\beta:[0, 1]\rightarrow
\overline{D^{\,\prime}}$ such that:

\medskip
a) $\alpha(t)\in D^{\,\prime}\setminus E$ and $\beta(t)\in
D^{\,\prime}\setminus E$ for $0\leqslant t<1,$

\medskip
b) $x_{k_{l_m}}\in |\alpha|,$ $y_{k_{l_m}}\in |\beta|,$
$m=1,2,\ldots ,$

\medskip
c) $\alpha(0)=a_1,$ $\beta(0)=a_2,$ $\alpha(1)=z_1,$ $\beta(1)=z_2,$

\medskip
d) $|\alpha|\cap |\beta|=\varnothing,$ see Figure~\ref{fig1}.
\begin{figure}[h]
\centerline{\includegraphics[scale=0.4]{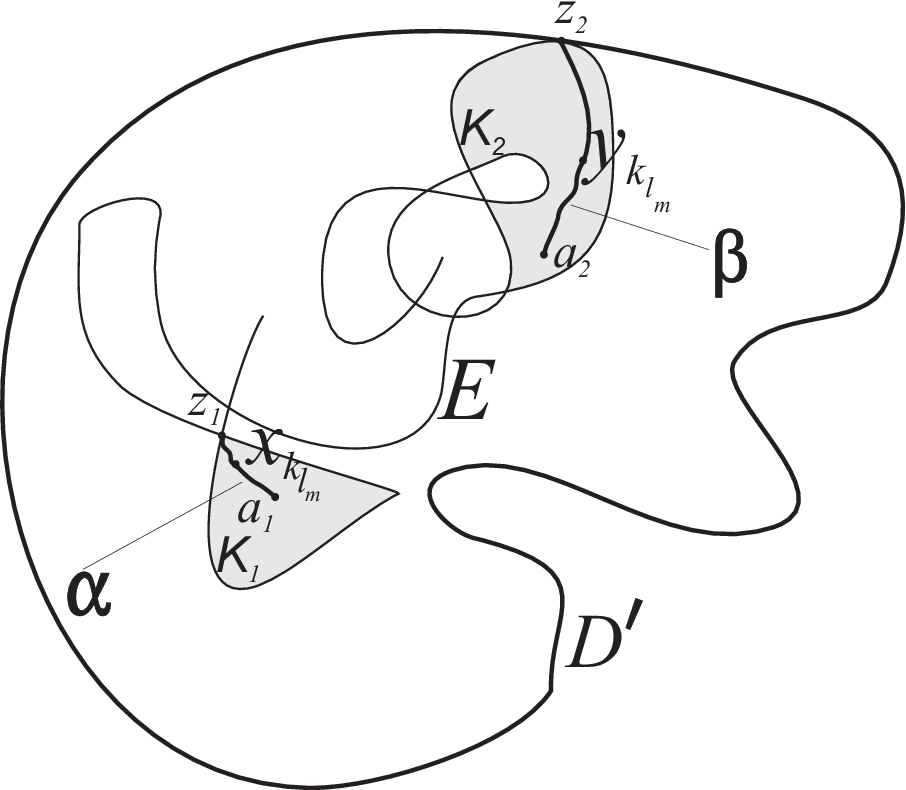}} \caption{The
formulation of Lemma~\ref{lem2}}\label{fig1}
\end{figure}}
\end{lemma}

\medskip
Given $\delta>0,$ domains $D, D^{\,\prime}\subset {\Bbb R}^n,$
$n\geqslant 2,$ $E\subset\overline{D^{\,\prime}},$ an element
$a\subset D^{\,\prime}\setminus E,$ a Lebesgue measurable function
$Q:{\Bbb R}^n\rightarrow [0, \infty],$ $Q(y)\equiv 0$ outside $D_0,$
we denote by $\frak{R}^{a}_{Q, \delta}(D, D^{\,\prime})$ is a family
of all open and discrete mappings $f:D\rightarrow D^{\,\prime}$ of a
domain $D$ onto $D^{\,\prime}$ such that (\ref{eq2*A})--(\ref{eqA2})
hold for every $y_0\in \overline{D^{\,\prime}},$ any
$0<r_1<r_2<\infty$ and, in addition, $h(f^{\,-1}(a),
\partial D)\geqslant \delta.$ The following statement is true,
cf. \cite[Lemma~4.1]{SevSkv$_3$},
\cite[Lemma~2.13]{ISS}.

\medskip
\begin{lemma}\label{lem1A}
{\it\, Let $\delta>0,$ let $D, D^{\,\prime}$ be domains in ${\Bbb
R}^n,$ $n\geqslant 2,$ let $a\in D^{\,\prime}$ and let $Q:{\Bbb
R}^n\rightarrow [0, \infty],$ $Q(y)\equiv 0$ outside $D^{\,\prime},$
be a Lebesgue measurable function. Assume that, no connected
component of the boundary of the domain $D$ degenerates into a point
and, besides that, at least one of the two conditions 1a) or 1b)
specified in Theorem~\ref{th2} is satisfied.

\medskip
If $E$ is closed in $D^{\,\prime}$ and $E_*$ is some (may be
degenerate) continuum in the component $K$ of $D^{\,\prime}\setminus
E$ which contains $a,$ then there exists $\delta_*>0$ such that
$h(f^{\,-1}(E_*),
\partial D)\geqslant \delta_*>0$ for every $f\in\frak{R}^{a}_{Q,
\delta}(D, D^{\,\prime}).$ }
\end{lemma}

\medskip
\begin{proof} Observe that, $f^{\,-1}(E_*)$ is a compact set in $D$
whenever $f\in\frak{R}^{a}_{Q, \delta}(D, D^{\,\prime}).$
Otherwise, there is $x_k\in f^{\,-1}(E_*)$ with $d(x_k,
\partial D)\rightarrow 0$ as $k\rightarrow\infty.$ Since $\overline{{\Bbb R}^n}$
is a compact metric space, we may consider that $f(x_k)$ converges
to some $b\in \overline{{\Bbb R}^n}$ as $k\rightarrow\infty.$ By the
definition, $b\in C(f, \partial D)\cap E_*\subset E\cap E_*,$ that
contradicts the definition of $E_*.$

\medskip
Therefore, the relation $h(f^{\,-1}(E),
\partial D)\geqslant \delta_*$ is obvious for some $\delta_*=\delta_*(f).$
The problem is only about the presence of a common $\delta_*>0$ that
provides the entire family of mappings $\frak{R}^{a}_{Q, \delta}(D,
D^{\,\prime}).$

\medskip
Let us prove Lemma~\ref{lem1A} by the contradiction. Assume that,
the conclusion of the lemma is not true. Then for each $m\in {\Bbb
N}$ there is some number $m\in {\Bbb N}$ such that
$h(f^{\,-1}_{m}(E_*),
\partial D)<1/m$ for some $f_m\in \frak{R}^{a}_{Q, \delta}(D,
D^{\,\prime}).$ Since $f^{\,-1}_{m}(E_*)$ is a compactum in $D$ for
each $m=1,2,\ldots,$ we have that $h(f^{\,-1}_{m}(E_*),
\partial D)=h(x_m, y_m)<\frac{1}{m}$ for some
$x_m\in f^{\,-1}_{m}(E_*)$ and $y_m\in \partial D.$ Since $\partial
D$ is a compact set, we may assume that $y_m\rightarrow y_0$ as
$m\rightarrow\infty$ for some $y_0\in
\partial D.$ Now, $x_m\rightarrow y_0$ as $m\rightarrow\infty,$ as
well.

\medskip
Let $w_m=f_m(x_m)\in E_*.$ We may consider that $w_m$ converges to
some point $w_0$ in $E_*$ as $m\rightarrow\infty.$ Observe that,
$w_m\in B(w_0, r_0)\subset D^{\,\prime}\setminus E$ for some $r_0>0$
and sufficiently large $m=1,2\ldots.$ We join the points $w_0$ and
$a$ by some path $\gamma:[1, 2]\rightarrow D^{\,\prime}\setminus E,$
$\gamma(1)=w_0,$ $\gamma(2)=a,$ and let $\gamma_m:[0, 1]\rightarrow
D^{\,\prime}$ be a segment $\gamma_m(t)=w_m+(w_0-w_m)t,$ $t\in [0,
1].$ Now, $K_0:=|\gamma|\cup\overline{B(w_0, r_0)}$ is a compactum
in $D^{\,\prime}.$ We set
$$E_m(t)=\begin{cases}\gamma_m(t)\,,& t\in [0, 1]\,\\
\gamma(t)\,,& t\in [1, 2]\end{cases}\,.$$
Now, $|E_m|$ is a compactum in $D^{\,\prime}\setminus E$ for all
$m=1,2,\ldots.$ Let $A_m:[0, c_m)\rightarrow D$ be a maximal
$f_m$-lifting of $E_m$ starting at $x_m$ (it exists by
Proposition~\ref{pr3}).

\medskip
Observe that, no path $A_m(t),$ $A_m:[0, c_m)\rightarrow D,$ cannot
tend to the boundary of the domain $D$ as $t\rightarrow c_m-0,$
because $C(f_m, \partial D)\subset E$ and simultaneously $|E_m|$ is
a compactum in $D_0.$ Now, by Proposition~\ref{pr3} $c_m=2$ and
$A_m$ has a limit as $t\rightarrow 2-0.$ Observe that,
$A_m(2):=z_m\in f_m^{\,-1}(E)$ by the definition of $f_m^{\,-1}(E).$

\medskip
Let $E_0$ be a component of $\partial D$ containing $y_0.$ Since by
the assumptions of the lemma, all components of $\partial D$ are
non-degenerate, there exists $r>0$ such that $h(E_0)\geqslant r.$
Put $P>0$ and $U=B_h(y_0, R_0)=\{y\in \overline{{\Bbb R}^n}: h(y,
y_0)<R_0\},$ where $2R_0:=\min\{r/2, \delta/2\}$ and $\delta$ is a
number from the definition of the class $\frak{F}^{a}_{Q, \delta}(D,
D_0)$. Observe that $A_m\cap U\ne\varnothing\ne A_m\setminus U$ for
sufficiently large $m\in{\Bbb N},$ since $x_m\rightarrow y_0$ as
$m\rightarrow \infty,$ $x_m\in A_m;$ besides that $h(A_m)\geqslant
\delta\geqslant 2R_0$ and $h(U)\leqslant 2R_0.$ Since $A_m$ is a
continuum, $A_m\cap
\partial U\ne\varnothing$ by Proposition~\ref{pr2}. Similarly,
$E_0\cap U\ne\varnothing\ne E_0\setminus U$ for sufficiently large
$m\in{\Bbb N},$ since $h(E_0)\geqslant r> 2R_0$ and $h(U)\leqslant
2R_0.$ Since $E_0$ is a continuum, $E_0\cap
\partial U\ne\varnothing$ by Proposition~\ref{pr2}.
By the proving above,
\begin{equation*}\label{eq8}
A_m\cap
\partial U\ne\varnothing\ne E_0\cap
\partial U\,.
\end{equation*}
By Lemma~\ref{lem3} there is $V\subset U,$ $V$ is a neighborhood of
$y_0,$ such that
\begin{equation}\label{eq9}
M(\Gamma(E, F, \overline{{\Bbb R}^n}))>P
\end{equation}
for any continua $E, F\subset \overline{{\Bbb R}^n}$ with $E\cap
\partial U\ne\varnothing\ne E\cap \partial V$ and $F\cap \partial
U\ne\varnothing\ne F\cap \partial V.$
Arguing similarly to above, we may prove that
$$
A_m\cap
\partial V\ne\varnothing\ne E_0\cap
\partial V
$$
for sufficiently large $m\in {\Bbb N}.$ Thus, by~(\ref{eq9})
\begin{equation}\label{eq9B}
M(\Gamma(A_m, E_0,\overline{{\Bbb R}^n}))>P
\end{equation}
for sufficiently large $m=1,2,\ldots .$ Let $\gamma:[0,
1]\rightarrow \overline{{\Bbb R}^n}$ be a path in $\Gamma(A_m, E_0,
\overline{{\Bbb R}^n}),$ i.e., $\gamma(0)\in A_m,$ $\gamma(1)\in
E_0$ and $\gamma(t)\in \overline{{\Bbb R}^n}$ for $t\in (0, 1).$ Let
$t_m=\sup\limits_{\gamma(t)\in D}t$ and let
$\alpha_m(t)=(\gamma|_{[0, t_m)})(t).$ Let $\Gamma_m$ consists of
all such paths $\alpha_m,$ now $\Gamma(A_m, E_0, \overline{{\Bbb
R}^n})>\Gamma_m$ and by the minorization principle of the modulus
(see~\cite[Theorem~1]{Fu})
\begin{equation}\label{eq12}
M(\Gamma_m)\geqslant M(\Gamma(A_m, E_0, \overline{{\Bbb R}^n}))\,.
\end{equation}
Combining~(\ref{eq9B}) and~(\ref{eq12}), we obtain that
\begin{equation}\label{eq9C}
M(\Gamma_m)>P
\end{equation}
for sufficiently large $m=1,2,\ldots .$

\medskip
We now prove that the relation~(\ref{eq9C}) contradicts the
definition of $f_m$ in~(\ref{eq2*A})--(\ref{eqA2}). Let
$\Delta_m:=f_m(\Gamma_m).$ By the definition, $C(f_m,
\partial D)\subset E.$ Thus, $C(\beta_m(t), t_m)\subset E,$
where $\beta_m(t)=f_m(\alpha_m(t))=(f_m(\gamma|_{[0, t_m)}))(t),$
$\gamma\in \Gamma(A_m, E_0, \overline{{\Bbb R}^n})$ and
$t_m=\sup\limits_{\gamma(t)\in D}t.$

\medskip
Obviously, ${\rm dist\,}(K_0, E)\geqslant \varepsilon_0>0$ for some
$\varepsilon_0>0$ and all sufficiently large $m=1,2,\ldots.$ Due to
the condition~1) of Theorem~\ref{th2}, for any
$y_0\in\overline{D^{\,\prime}}$ we may find
$\varepsilon=\varepsilon(y_0)>0,$ such that
$\varepsilon/4<\varepsilon/2<\varepsilon_0$ and the relation
$\int\limits_{\varepsilon/4}^{\varepsilon/2}\
\frac{dr}{r\widetilde{q}_{y_0}^{\frac{1}{n-1}}(r)}>0$ holds, where
$\widetilde{q}_{y_0}(r)$ is defined in~(\ref{eq2B}).
We consider now the cover of the set $K_0$ of the following type:
$\bigcup\limits_{y_0\in K_0}B(y_0, \varepsilon(y_0)/4).$ Since $K_0$
is a compactum in $D^{\,\prime},$ there is a number $M_0\in {\Bbb
N}$ such that $K_0\subset \bigcup\limits_{i=1}^{M_0} B(z_i,
\varepsilon^{\,i}/4),$ where $\varepsilon^{\,i}=\varepsilon(z_i),$
$z_i\in K_0$ for $1\leqslant i\leqslant M_0.$ Observe that
\begin{equation}\label{eq6D}
\Delta_m=\bigcup\limits_{i=1}^{M_0}\Delta_{mi}\,,
\end{equation}
where $\Delta_{mi}$ consists of all paths $\gamma:[0, 1)\rightarrow
D^{\,\prime}$ in $\Delta_m$ such that $\gamma(0)\in B(y_i,
\varepsilon^{\,i}/4).$ We now show that
\begin{equation}\label{eq7C}
\Delta_{mi}>\Gamma(S(y_i, \varepsilon^{\,i}/4), S(y_i,
\varepsilon^{\,i}/2), A(y_i, \varepsilon^{\,i}/4,
\varepsilon^{\,i}/2))\,.
\end{equation}
Indeed, let $\Delta\in \Gamma_{mi},$ in other words, $\gamma:[0,
1)\rightarrow D^{\,\prime},$ $\gamma\in \Delta_m$ and $\gamma(0)\in
B(y_i, \varepsilon^{\,i}/4).$ Now, by the definition of
$\varepsilon^{\,i},$ $|\gamma|\cap B(y_i,
\varepsilon^{\,i}/4)\ne\varnothing\ne |\gamma|\cap
(D^{\,\prime}\setminus B(y_i, \varepsilon^{\,i}/4)).$ Therefore, by
Proposition~\ref{pr2} there is $0<t_1<1$ such that $\gamma(t_1)\in
S(y_i, \varepsilon^{\,i}/4).$ We may assume that $\gamma(t)\not\in
B(y_i, \varepsilon^{\,i}/4)$ for $t>t_1.$ Put
$\gamma_1:=\gamma|_{[t_1, 1]}.$ Similarly, $|\gamma_1|\cap B(y_i,
\varepsilon^{\,i}/2)\ne\varnothing\ne |\gamma_1|\cap
(D^{\,\prime}\setminus B(y_i, \varepsilon^{\,i}/2)).$ By
Proposition~\ref{pr2} there is $t_1<t_2<1$ with $\gamma(t_2)\in
S(y_i, \varepsilon^{\,i}/2).$ We may assume that $\gamma(t)\in
B(y_i, \varepsilon^{\,i}/2)$ for $t<t_2.$ Put
$\gamma_2:=\gamma|_{[t_1, t_2]}.$ Then, the path $\gamma_2$ is a
subpath of $\gamma,$ which belongs to the family $\Gamma(S(y_i,
\varepsilon^{\,i}/4), S(y_i, \varepsilon^{\,i}/2), A(y_i,
\varepsilon^{\,i}/4, \varepsilon^{\,i}/2)).$ Thus, the
relation~(\ref{eq7C}) is established. By~(\ref{eq6D}) and
(\ref{eq7C}), since $\Delta_m:=f_m(\Gamma_m),$ we obtain that
\begin{equation}\label{eq5}
\Gamma_m>\bigcup\limits_{i=1}^{M_0}\Gamma_{mi}\,,
\end{equation}
where $\Gamma_{mi}:=\Gamma_{f_m}(y_i, \varepsilon^{\,i}/4,
\varepsilon^{\,i}/2).$ Set $\widetilde{Q}(y)=\max\{Q(y), 1\}$ and
$$\widetilde{q}_{y_i}(r)=\int\limits_{S(y_i,
r)}\widetilde{Q}(y)\,d\mathcal{A}\,.$$
Now, arguing similarly to the proof of~(\ref{eq2A}) and~(\ref{eq3}),
we obtain that
\begin{equation}\label{eq7D} M(\Gamma_m)\leqslant
\sum\limits_{i=1}^{M_0}M(\Gamma_{mi})\leqslant
\sum\limits_{i=1}^{M_0}\frac{\omega_{n-1}}{I_i^{n-1}}:=C_0\,, \quad
m=1,2,\ldots\,,
\end{equation}
where $I_i=I_i(y_i, \varepsilon^{\,i}/4,
\varepsilon^{\,i}/2)=\int\limits_{\varepsilon^{\,i}/4}^{\varepsilon^{\,i}/2}\
\frac{dr}{r\widetilde{q}_{y_i}^{\frac{1}{n-1}}(r)}\,.$
Since $P$ in~(\ref{eq9C}) may be done arbitrary large, the
relations~(\ref{eq9C}) and~(\ref{eq7D}) contradict each other. This
completes the proof.~$\Box$
\end{proof}

\section{Proof of Theorem~\ref{th2}}

Let us prove Theorem~\ref{th2} by the contradiction, namely,  assume
that there is $x_0\in \partial D,$ a number $\varepsilon_0>0,$ the
sequences $x_m, y_m\in D\setminus E_0,$ which converge to the point
$x_0$ as $m\rightarrow\infty,$ and the corresponding maps $f_m\in
{\frak S}^P_{E, E_0 \delta, Q}(D, D^{\,\prime})$ such that
\begin{equation}\label{eq12A}
\rho(f_m(x_m),f_m(y_m))\geqslant\varepsilon_0,\quad m=1,2,\ldots .
\end{equation}
Using the M\"{o}bius transformation $\varphi:\infty\mapsto 0$ if
necessary, due to the invariance of the modulus $M$ in~(\ref{eq2*A})
(see~\cite[Theorem~8.1]{Va}), we may assume that $x_0\ne\infty.$
Since the space $(\bigcup\limits_{i=1}^N\overline{D_i}_P, \rho)$ is
compact, we may assume that $f_m(x_m)$ and $f_m(y_m)$ converge as
$m\rightarrow\infty$ to some points $P_1$ and $P_2.$ By the
continuity of the metric $\rho,$ (\ref{eq12A}) implies that $P_1\ne
P_2.$ By the assumption~2), $D^{\,\prime}\setminus E$ consists of
finite number of domains $D_1, D_2,\ldots, D_N$ any of which is
regular. Consequently, there are domains $D_{j_1}$ and $D_{j_2,}$
$1\leqslant j_1, j_2\leqslant N,$ and subsequences
$f_{m_k}(x_{m_k}), f_{m_k}(y_{m_k}),$ $k=1,2,\ldots,$ of $f_m(x_m)$
and $f_m(y_m),$ such that $f_{m_k}(x_{m_k})\in D_{j_1}$ and
$f_{m_k}(y_{m_k})\in D_{j_2}$ for $k\in {\Bbb N}.$ Without loss of
generality, making a relabeling if necessary, we may assume that
$f_m(x_m)\in D_{j_1}$ and $f_m(y_m)\in D_{j_2},$ $i=1,2,\ldots .$
Here the case $j_1=j_2$ is possible. By the definition of the metric
$\rho$ in~(\ref{eq1}) we have that $\rho_{j_1}(f_m(x_m),
P_1)\rightarrow 0$ and $\rho_{j_2}(f_m(y_m), P_2)\rightarrow 0$  as
$m\rightarrow\infty,$ where $\rho_{j_1}$ and $\rho_{j_2}$ are
corresponding metrics in spaces $\overline{D}_{j_1}$ and
$\overline{D}_{j_2},$ respectively.

\medskip
Let $z_0:=a_{j_1}\in D_{j_1}\in P$ and $y_0:=a_{j_2}\in D_{j_2}\in
P$ (such points exist by the condition of Theorem~\ref{th2}). If
$j_1=j_2,$ then $z_0\ne y_0$ by the definition. Otherwise, if
$j_1=j_2,$ we also may assume that $z_0\ne y_0$ by
Lemma~\ref{lem1A}. Let $d_m$ and $g_m$ be sequences of decreasing
domains corresponding to the prime ends $P_1$ and $P_2,$
respectively. By~\cite[Lemma~3.1]{IS}, cf.~\cite[Lemma~1]{KR}, we
may assume that the sequence of cuts $\sigma_m,$ corresponding to
domains~$d_m,$ $m=1,2,\ldots, $ lies on the
spheres~$S(\overline{x_0}, r_m),$ where $\overline{x_0}\in
\partial D^{\,\prime}$ and $r_m\rightarrow 0$ as $m\rightarrow\infty.$
Without loss of generality, we may assume that $d_1\cap
g_1=\varnothing$ and $x_0, y_0\not\in d_1\cup g_1.$

\medskip
By Lemmas~\ref{lem1} and~\ref{lem4} there exist disjoint paths
$\gamma_{1,m}:[0, 1]\rightarrow D^{\,\prime}$ and $\gamma_{2,m}:[0,
1]\rightarrow D^{\,\prime}$ and the number $N>0$ such that
$\gamma_{1, m}(0)=f_m(x_m),$ $\gamma_{1, m}(1)=z_0,$ $\gamma_{2,
m}(0)=f_m(y_m),$ $\gamma_{2, m}(1)=y_0$ and
\begin{equation}\label{eq15}
M(\Gamma_m)\leqslant N\,,\quad m\geqslant M_0
\end{equation}
for some $M_0>0,$ where $\Gamma_m$ consists of those and only those
paths $\gamma$ in $D$ for which $f_m(\gamma)\in\Gamma(|\gamma_{1,
m}|, |\gamma_{2, m}|, D^{\,\prime})$ (see Figure~\ref{fig6}).
\begin{figure}
\centering\includegraphics[width=300pt]{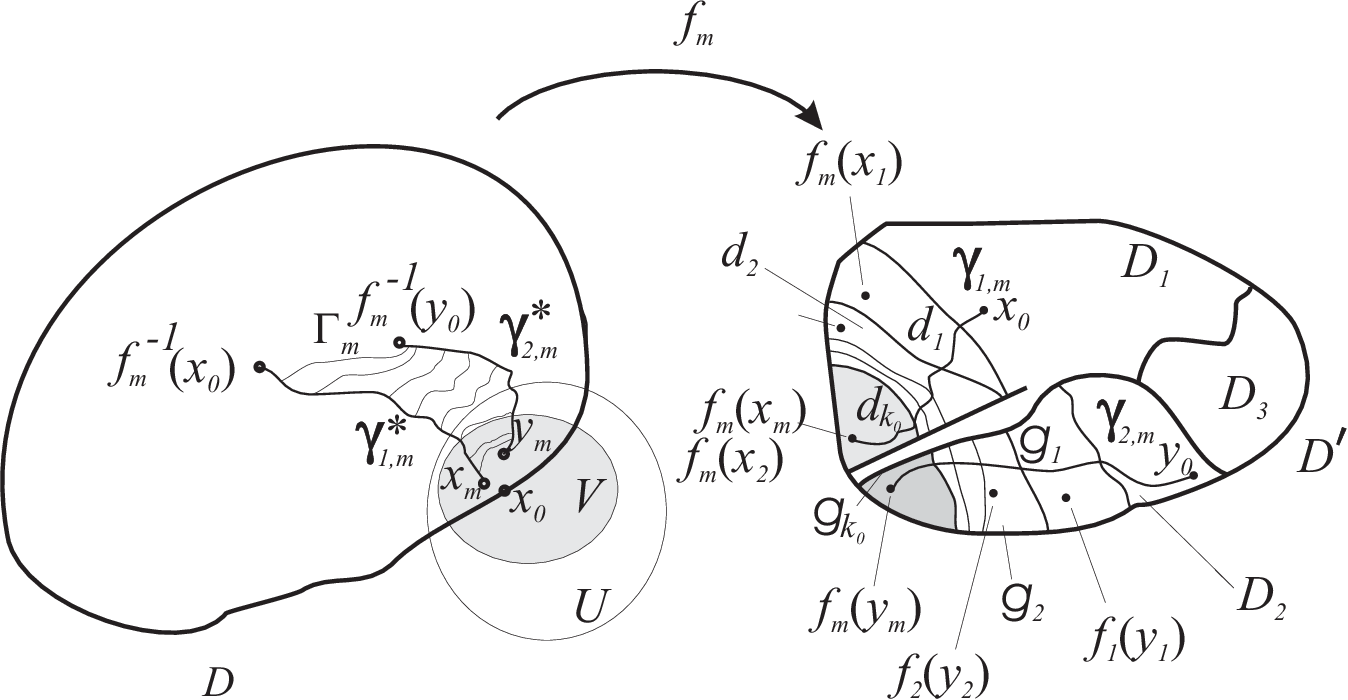} \caption{To
the proof of theorem~\ref{th2}.}\label{fig6}
\end{figure}
On the other hand, let~$\gamma^*_{1,m}:[0, c^1_m)\rightarrow D$
and~$\gamma^*_{2,m}:[0, c^2_m)\rightarrow D$ be maximal
$f_m$-liftings of the paths~$\gamma_{1,m}$ and~$\gamma_{2,m}$ with
origins at the points~$x_m$ and $y_m,$ respectively (such liftings
exist by Proposition~\ref{pr3}). By the same proposition, one of two
situations is possible:

\medskip
1) $\gamma^*_{1,m}(t)\rightarrow x_1\in D$ as $t\rightarrow c^1_m-0$
and $c^1_m=1$ and $f_m(\gamma^*_{1,m}(1))=z_0,$

\medskip
2) or $\gamma^*_{1,m}(t)\rightarrow \partial D$ as $t\rightarrow
c^1_m.$

\medskip
Let us show that situation~2) does not hold. Assume the opposite:
let $\gamma^*_{1,m}(t)\rightarrow \partial D$ as $t\rightarrow
c^1_m.$ We choose an arbitrary sequence $\theta_k\in [0, c^1_m)$
such that $\theta_k\rightarrow c^1_m-0$ as $k\rightarrow\infty.$
Since the space $\overline{{\Bbb R}^n}$ is compact, $\partial D$ is
also compact as a closed subset of a compact space. Then there
exists $w_k\in
\partial D$ is such that
\begin{equation}\label{eq7C_2}
h(\gamma^*_{1,m}(\theta_k), \partial D)=h(\gamma^*_{1,m}(\theta_k),
w_k) \rightarrow 0\,,\qquad k\rightarrow \infty\,.
\end{equation}
Due to the compactness of $\partial D$, we may assume that
$w_k\rightarrow w_0\in \partial D$ as $k\rightarrow\infty.$
Therefore, from the relation~(\ref{eq7C_2}) and by the triangle
inequality
\begin{equation}\label{eq8C}
h(\gamma^*_{1,m}(\theta_k), w_0)\leqslant
h(\gamma^*_{1,m}(\theta_k), w_k)+h(w_k, w_0)\rightarrow 0\,,\qquad
k\rightarrow \infty\,.
\end{equation}
On the other hand,
\begin{equation}\label{eq9C_2}
f_m(\gamma^*_{1,m}(\theta_k))\rightarrow\gamma_{1, m}(c^1_m)\,,\quad
k\rightarrow\infty\,,
\end{equation}
because by the construction the whole path $\gamma_{1, m}$ lies in
$D^{\,\prime}\setminus E\subset D^{\,\prime}\setminus C(f_m,\partial
D)$ together with its endpoints. At the same time, by~(\ref{eq8C})
and~(\ref{eq9C_2}) we have that $\gamma_{1, m}(c^1_m)\in C(f_m,
\partial D).$ The inclusion $|\gamma_{1, m}|\subset
D^{\,\prime}\setminus E$ contradicts the relation $\gamma_{1,
m}(c^1_m)\in E.$ The resulting contradiction indicates the
impossibility of the second situation above. Arguing similarly we
may conclude that $\gamma^*_{2,m}(t)\rightarrow x_2\in D$ as
$t\rightarrow c^2_m-0$ and $c^2_m=1$ and
$f_m(\gamma^*_{2,m}(1))=f_m(x_2)=a_{j_2}=y_0.$

\medskip
By the definition of the class ${\frak S}^P_{E, E_0 \delta, Q}(D,
D^{\,\prime}),$ we obtain that $h(f_m^{\,-1}(z_0), \partial
D)\geqslant \delta$ and $h(f_m^{\,-1}(y_0), \partial D)\geqslant
\delta$ (where $h(A, B)$ denotes the chordal distance between sets
in $\overline{{\Bbb R}^n}$). Now
$$h(|\gamma^*_{1, m}|)\geqslant h(x_m, \gamma^*_{1,m}(1)) \geqslant
(1/2)\cdot h(\gamma^*_{1,m}(1), \partial D)>\delta/2\,,$$
\begin{equation*}\label{eq14}
h(|\gamma^*_{2, m}|)\geqslant h(y_m, \gamma^*_{2,m}(1)) \geqslant
(1/2)\cdot h(\gamma^*_{2,m}(1), \partial D)>\delta/2
\end{equation*}
for sufficiently large $m\in {\Bbb N}.$
We choose the ball $U:=B_h(x_0, r_0)=\{x\in\overline{{\Bbb R}^n}:
h(x, x_0)<r_0\},$ where $r_0>0$ and $r_0<\delta/4.$ Note that
$|\gamma^*_{1, m}|\cap U\ne\varnothing\ne |\gamma^*_{1, m}|\cap
(D\setminus U)$ for sufficiently large $m\in{\Bbb N},$ because
$h(|\gamma^*_{1, m}|)\geqslant \delta/2$ and $x_m\in|\gamma^*_{1,
m}|,$ $x_m\rightarrow x_0$ as $m\rightarrow\infty.$ Reasoning
similarly, we may conclude that~$|\gamma^*_{2, m}|\cap
U\ne\varnothing\ne |\gamma^*_{2, m}|\cap (D\setminus U).$ Since
$|\gamma^*_{1, m}|$ and $|\gamma^*_{2, m}|$ are continua, due to
~\cite[Theorem~1.I.5.46]{Ku$_2$}
\begin{equation}\label{eq8AA}
|\gamma^*_{1, m}|\cap \partial U\ne\varnothing, \quad |\gamma^*_{2,
m}|\cap
\partial U\ne\varnothing\,.
\end{equation}
Fix $P:=N>0,$ where $N$ is a number from the relation~(\ref{eq15}).
Since the boundary of the domain $D$ is weakly flat, there is a
neighborhood $V\subset U$ of the point $x_0,$ such that the
inequality
\begin{equation}\label{eq9AA}
M(\Gamma(E, F, D))>N
\end{equation}
holds for any continua $E, F\subset D$ with $E\cap
\partial U\ne\varnothing\ne E\cap \partial V$ and $F\cap \partial
U\ne\varnothing\ne F\cap \partial V.$
By~\cite[Theorem~1.I.5.46]{Ku$_2$}
\begin{equation}\label{eq10AA}
|\gamma^*_{1, m}|\cap \partial V\ne\varnothing, \quad |\gamma^*_{2,
m}|\cap
\partial V\ne\varnothing\end{equation}
for sufficiently large $m\in {\Bbb N}.$ By~(\ref{eq8AA}),
(\ref{eq9AA}) and (\ref{eq10AA}), we obtain that
\begin{equation}\label{eq6a}
M(\Gamma(|\gamma^*_{1, m}|, |\gamma^*_{2, m}|, D))>N\,.
\end{equation}
The inequality~(\ref{eq6a}) contradicts~(\ref{eq15}), because
$\Gamma(|\gamma^*_{1, m}|, |\gamma^*_{2, m}|, D)\subset \Gamma_m,$
therefore,
$$M(\Gamma(|\gamma^*_{1, m}|, |\gamma^*_{2, m}|, D))
\leqslant M(\Gamma_m)\leqslant N\,.$$
The resulting contradiction indicates that the assumption
in~(\ref{eq12A}) is false. The theorem is proved.~$\Box$

\section{On Orlicz-Sobolev classes}

We now formulate a result for Orlicz-Sobolev classes $W^{1,
\varphi}_{\rm loc},$ where $\varphi$ satisfies Calderon condition
$\int\limits_{1}^{\infty}\left(\frac{t}{\varphi(t)}\right)^
{\frac{1}{n-2}}dt<\infty.$ For definitions of Orlicz-Sobolev classes
see, e.g., \cite{Sev$_2$}.

\medskip
Given number $\delta>0,$ $K\geqslant 1,$ $N_0\in {\Bbb N},$ a
non-decreasing function $\varphi:(0,\infty)\rightarrow (0,\infty),$
domains $D, D^{\,\prime}\subset{\Bbb R}^n,$ $n\geqslant 3,$ a set
$E\subset \overline{D^{\,\prime}},$ closed in $\overline{{\Bbb
R}^n}$ such that $\partial D^{\,\prime}\subset E,$ a set $E_0\subset
D,$ a finite set of points $P=\{a_i\}_{i=1}^{S}\subset
D^{\,\prime}\setminus E,$ $0\leqslant i\leqslant S,$ and a Lebesgue
measurable function $Q:{\Bbb R}^n\rightarrow [0, \infty],$ denote by
${\frak OS}^{P, K, \varphi}_{E, E_0, \delta, Q}(D, D^{\,\prime})$ a
family of all open discrete mappings $f$ of $D$ onto $D^{\,\prime}$
such that $f= g\circ\psi,$ where $\psi$ is some $K$-quasiregular
mapping for which $N(\psi, D)\leqslant N_0,$ and $g$ is a
homeomorphism in $\psi(D)$ such that $g^{\,-1}\in W^{1,
\varphi}_{\rm loc}(D^{\,\prime})$ and

\medskip
1) $C(f, \partial D)\subset E,$

\medskip
2) $f^{\,-1}(E\cap D^{\,\prime})\subset E_0,$

\medskip
3) $h(f^{\,-1}(a_i), \partial D)\geqslant \delta$ for all $i\in
{\Bbb N}$ (here $h(A, B)$ denotes the chordal distance between sets
in $\overline{{\Bbb R}^n},$ $h(A, B)=\inf\limits_{x\in A, y\in
B}h(x, y)$),

\medskip
4) $K^{n-1}_O(x, f)\leqslant Q(x)$ a.e., where $$K_{O}(x,f)\quad
=\quad \left\{
\begin{array}{rr}
\frac{\Vert f^\prime(x)\Vert^n}{|J(x,f)|}, & J(x,f)\ne 0,\\
1,  &  f^{\,\prime}(x)=0, \\
\infty, & {\rm otherwise}
\end{array}
\right.\,.$$
The following statement holds.

\medskip
\begin{theorem}\label{th1}
{\it\, Assume that, under conditions and notions of
Theorem~\ref{th2}, $Q_{\rm loc}^1(D)$ and
\begin{equation*}\label{eqOS3.0a}
\int\limits_{1}^{\infty}\left(\frac{t}{\varphi(t)}\right)^
{\frac{1}{n-2}}dt<\infty\,.
\end{equation*}
Then the family ${\frak OS}^{P, K, \varphi}_{E, E_0, \delta, Q}(D,
D^{\,\prime})$ is equicontinuous at every point $x_0\in \partial D$
with respect to $D\setminus E_0,$ i.e., for every $\varepsilon>0$
there is $\delta=\delta(x_0, \varepsilon)>0$ such that $\rho(f(x),
f(y))<\varepsilon$ for any $x, y\in B(x_0, \delta)\cap (D\setminus
E_0)$ and any $f\in{\frak OS}^{P, K, \varphi}_{E, E_0, \delta, Q}(D,
D^{\,\prime}),$ where the metric $\rho$ is defined in~(\ref{eq1}).
In particular, any $f\in {\frak OS}^{P, K, \varphi}_{E, E_0, \delta,
Q}(D, D^{\,\prime})$ has a continuous extension to $x_0$ with
respect to $D\setminus E_0$ in terms of the space
$\bigl(\bigcup\limits_{i=1}^N\overline{D_i}_P, \rho\bigr).$ }
\end{theorem}

\medskip
\begin{proof}
Let $f\in{\frak OS}^{P, K, \varphi}_{E, E_0, \delta, Q}(D,
D^{\,\prime}).$ Now, there is a $K$-quasiregular mapping $\psi$ with
$N(\psi, D)\leqslant N_0$ and a homeomorphism $g$ in $D$ with
$g^{\,-1}\in W^{1, \varphi}_{\rm loc}(D^{\,\prime})$ such that $f=
g\circ \psi.$ By \cite[Lemma~4.9]{Sev$_2$}, $g$
satisfies~(\ref{eq2*A})--(\ref{eqA2}) at any point $y_0\in
\overline{D^{\,\prime}}.$ Since $\psi$ is $K$-quasiregular, $\psi$
satisfies~(\ref{eq22}) with $K_0=K.$ Thus, $f$
satisfies~(\ref{eq2*A})--(\ref{eqA2}) at any point $y_0\in
\overline{D^{\,\prime}}$ with $\widetilde{Q}=Q\cdot N_0\cdot K.$ The
desired conclusion follows from Theorem~\ref{th2}.~$\Box$
\end{proof}

\section{Examples}

\begin{example}\label{ex1}
Let $D={\Bbb D}$ be the unit disk. By the Riemannian mapping
theorem, there exists a conformal mapping $f_1$ of $D$ onto
$D_1:=\Pi\setminus\bigcup\limits_{k=2}^{\infty}I_k,$ where $\Pi=(0,
1)\times (0, 4\pi)$ and $I_k=\{z=(x, y)\in {\Bbb R}^2:
x=1/k,\,\,0<y\leqslant 1/2\},$ $k=2,3,\ldots$ (see
Figure~\ref{fig7}).
\begin{figure}
  \centering\includegraphics[scale=0.45]{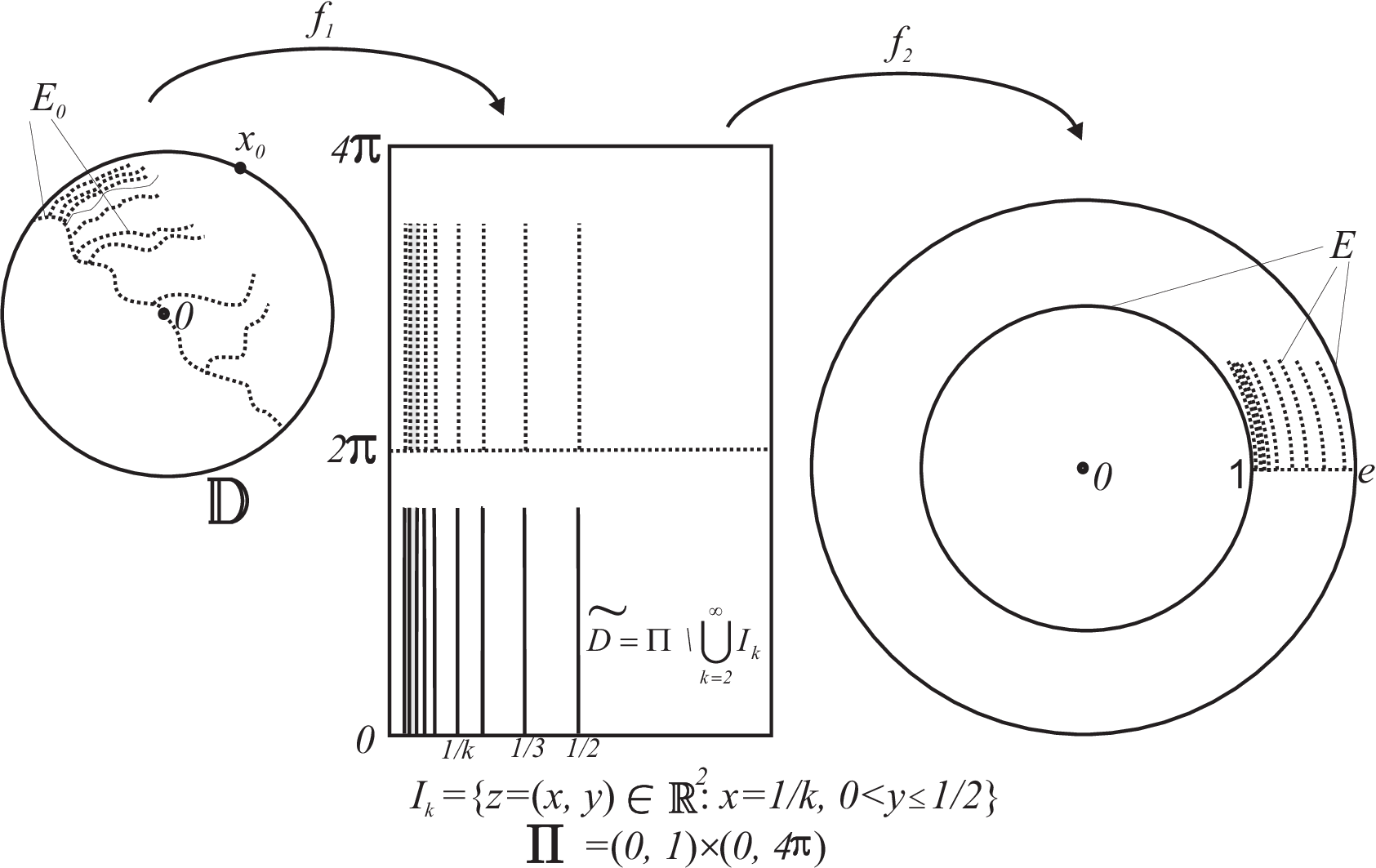}
  \caption{Illustration for the Example~\ref{ex1}}\label{fig7}
 \end{figure}
The mapping $f_2(z)=e^z$ maps $D_1$ onto the annulus
$D_2=\{1<|z|<e\},$ while $C(f_2, \partial D_1)=S(0, 1)\cup S(0,
e)\cup J_0\cup\bigcup\limits_{k=2}^{\infty}J_k,$ where
$J_0=\{z=x+iy: y=0, 1\leqslant x\leqslant e\}$ and $J_k=\{z\in {\Bbb
C}: z=re^{i\varphi}, r=e^{1/k}, 0\leqslant \varphi\leqslant 1/2\}.$
Set $E:=C(f_2, \partial D_1)$ and $D^{\,\prime}=D_2.$ Now $E\cap
D^{\,\prime}=J_0\cup\bigcup\limits_{k=2}^{\infty}J_k.$ We set
$f=f_2\circ f_1.$ Observe that, $E$ does not split $D^{\,\prime},$
$D^{\,\prime}$ is a regular domain and $f^{\,-1}_2(E\cap
D^{\,\prime})=L_0\cup \bigcup\limits_{k=2}^{\infty}L_k,$ where
$L_0=\{z=x+yi: y=2\pi, 0\leqslant x\leqslant 4\}$ and $L_k=\{z=(x,
y)\in {\Bbb R}^2: x=1/k,\,\,2\pi<y\leqslant 2\pi+1/2\},$
$k=2,3,\ldots .$ Now, $E_0:=f^{\,-1}_1(f^{\,-1}_2(E\cap
D^{\,\prime}))=f^{\,-1}(E\cap D^{\,\prime})$ is nowhere dense in $D$
because by \cite[Theorem~17.15]{Va} $f_1^{\,-1}$ has a continuous
extension to any point $y_0\in\overline{D_1}\setminus
(\bigcup\limits_{k=2}^{\infty}I_k\cup I_0),$ $I_0=\{z=x+iy: y=0,
0\leqslant x\leqslant 1\}.$ Consequently, $f_1^{\,-1}(L_0)$ and
$f_1^{\,-1}(L_k),$ $k=2,3, \ldots, $ are rectifiable paths,
therefore, their loci are nowever dense. Observe that, $f$ satisfies
the relations~(\ref{eq2*A})--(\ref{eqA2}) with $Q\equiv 1,$ see
Remark~\ref{rem1}. In turn, the function $Q$ satisfies the
relation~(\ref{eq6A}). Now, $f\in {\frak S}^P_{E, E_0 \delta, Q}(D,
D^{\,\prime})$ for $D,$ $D^{\,\prime},$ $Q,$ $E$ and $E_0$ mentioned
above. There are trivially exit $P=\{a_1\}$ and $\delta>0,$ where
$a_1\in D^{\,\prime}\setminus E$ may be chosen in arbitrary way, and
$\delta:=h(f^{\,\-1}(a_1), \partial {\Bbb D}).$ Finally, ${\Bbb D}$
has a weakly flat boundary (see \cite[Theorem~17.12]{Va}). All of
the conditions of Theorem~\ref{th2} hold for a family
$\{f\}=\widetilde{{\frak S}}^P_{E, E_0 \delta, Q}(D,
D^{\,\prime})\subset {\frak S}^P_{E, E_0 \delta, Q}(D,
D^{\,\prime})$ containing the unique mapping $f.$
\end{example}

\medskip
\begin{example}
Now, we construct a sequence $f_m,$ $m=1,\ldots,$ which forms some a
family of mappings ${\frak S}^P_{E, E_0, \delta, Q}(D,
D^{\,\prime})$ satisfying all conditions of Theorem~\ref{th2}. We
fix some a point $z_0\in D^{\,\prime}\setminus E$ and $0<r_0<{\rm
dist\,}(z_0, E\cup
\partial D^{\,\prime}).$ Set
$h(y)=\frac{r_0(y-z_0)}{|y-z_0|\log\frac{er_0}{|y-z_0|}}+z_0,$
$h(z_0)=z_0,$ $h(B(z_0, r_0))=B(z_0, r_0).$ Besides that, we set
$$h_m(y)=\begin{cases}h(z)\,,&1/m<|y-z_0|<r_0\,,\\
\frac{mr_0(y-z_0)}{\log(mer_0)}+z_0\,, &\text{otherwise\,\, in\,\,}
B(z_0, r_0)
\end{cases}\,\,,$$
Then $h^{\,-1}_m$ are well-defined in $B(z_0, r_0)$ and
$h^{\,-1}_m(B(z_0, r_0))=B(z_0, r_0).$ We define $h_m(z)=z$ in
$D^{\,\prime}\setminus B(z_0, r_0);$ observe that $h_m$ are
homeomorphisms in $D^{\,\prime}$ for each $m=1,2,\ldots.$ Let
$f_m:=h_m^{\,-1}\circ f.$ Reasoning similarly
to~\cite[Proposition~6.3]{MRSY}, we may show that $f_m$ satisfy the
relations~(\ref{eq2*A})--(\ref{eqA2}) in each point $z_0\in
\overline{D^{\,\prime}}$ for
\begin{equation}\label{eq5B}
Q=Q(y)=\begin{cases}\log\left(\frac{e}{r_0|y-z_0|}\right)\,, & z\in B(z_0, r_0)\,,\\
1\,,& z\in D^{\,\prime}\setminus  B(z_0, r_0)\,.\end{cases}
\end{equation}
Note that, $Q\in L^1(D^{\,\prime}).$ Indeed, by the Fubini theorem,
we have that
$$\int\limits_{B(z_0, r_0)}Q(y)\,dm(y)=\int\limits_0^1\int\limits_{S(z_0, r)}
\log\left(\frac{e}{|y-z_0|}\right)d\mathcal{H}^{1}(y)dr=$$
$$=2\pi\int\limits_0^1
r\log\left(\frac{e}{r}\right)\,dr\leqslant 2\pi
e\int\limits_0^1dr=2\pi e<\infty\,.$$
Obviously, $Q$ is integrable in $D^{\,\prime},$ as well. We observe
that the family $f_m,$ $m=1,2,\ldots,$ is containing in ${\frak
S}^P_{E, E_0, \delta, Q}(D, D^{\,\prime})$ for $D,$ $D^{\,\prime},$
$E$ and $E_0$ mentioned in Example~\ref{ex1}, while $Q$ is defined
in~(\ref{eq5B}). Indeed, since any $f_m,$ $m=1,2,\ldots,$ fix
infinitely many points in ${\Bbb D},$ we may choose some an element
$a_1\in D^{\,\prime}\setminus E$ and $\delta>0$ such that
$h(f_m^{\,-1}(a_1), \partial {\Bbb D})\geqslant \delta$ for all
$m\in {\Bbb N}.$ Set $P:=\{a_1\}.$ Thus, all $f_m\in{\frak S}^P_{E,
E_0, \delta, Q}(D, D^{\,\prime})$ for all $m=1,2,\ldots .$ Besides
that, all of conditions of Theorem~\ref{th2} hold.
\end{example}

Results for a continuous extension of mappings studied here were
published in \cite{SDK}. This paper is published in the preprint
form \cite{KS}.

\medskip
{\bf Declarations.}

\medskip
{\bf Funding.} The work was supported by the National Research
Foundation of Ukraine (Project ``Analogues of Carath\'{e}odory and
Koebe-Bloch theorems for Orlycz-Sobolev classes'', Project number
2025.02/0010).

\medskip
{\bf Conflicts of interest.} The author has no financial or
proprietary interests in any material discussed in this article.

\medskip
{\bf Availability of data and material.} The datasets generated
and/or analysed during the current study are available from the
corresponding author on reasonable request.

\medskip
{\bf \noindent Zarina Kovba} \\
Zhytomyr Ivan Franko State University,  \\
40 Velyka Berdychivs'ka Str., 10 008  Zhytomyr, UKRAINE \\
mazhydova@gmail.com

\medskip
\medskip
{\bf \noindent Evgeny Sevost'yanov} \\
{\bf 1.} Zhytomyr Ivan Franko State University,  \\
40 Velyka Berdychivs'ka Str., 10 008  Zhytomyr, UKRAINE \\
{\bf 2.} Institute of Applied Mathematics and Mechanics\\
of NAS of Ukraine, \\
19 Henerala Batyuka Str., 84 116 Slov'yansk,  UKRAINE\\
esevostyanov2009@gmail.com

\end{document}